\documentclass[a4paper, 12pt]{article}
\usepackage{cmap}			
\usepackage[T2A]{fontenc}	
\usepackage[utf8]{inputenc}
\usepackage[english, russian]{babel}
\usepackage{amsmath,amssymb} 
\usepackage[unicode, pdftex]{hyperref}
\usepackage{amsfonts}
\usepackage{mathtext} 
\usepackage{scrextend}
\usepackage[usenames]{color}
\usepackage{colortbl}
\usepackage{graphicx} 
\usepackage{amsthm}
\usepackage{fullpage}

\newtheorem{lemma}{Лемма}
\newtheorem{theorem}{Теорема}
\newtheorem{definition}{Определение}

\theoremstyle{remark}
\newtheorem{remark}{\bf Замечание}

\newenvironment{Proof} {\begin{proof}[{\bf Доказательство}]} {\end{proof}}

\DeclareMathOperator{\rank}{rank}

\newcommand{\R}{\mathbb{R}}

\newcommand{\cP}{{\cal P}}
\newcommand{\cD}{{\cal D}}
\newcommand{\cI}{{\cal I}}

\newcommand{\cB}{{\cal B}}
\newcommand{\cH}{{\cal H}}
\newcommand{\cF}{{\cal F}}

\date{}
\usepackage{cite}
\begin{document}
\renewcommand{\refname}{\normalsize Литература}

\title{
Локальная идентифицируемость параметра-функции в системе дифференциальных уравнений
\footnote{Исследование выполнено в Санкт-Петербургском международном математическом институте имени Леонарда Эйлера при финансовой поддержке Министерства науки и высшего образования Российской Федерации (соглашение № 075–15–2025–343 от 29.04.2025).}}
\author{В.\,С. Шалгин}

\maketitle

\vspace{-0.5cm}

{\footnotesize
\begin{tabbing}
    \= Санкт-Петербургский государственный университет, \=\\Россия, 199034, Санкт-Петербург, Университетская наб. 7/9, \+\\vladimir.shalgin@spbu.ru, st086496@student.spbu.ru
\end{tabbing}
}

{\small
\begin{quote}
\noindent{\sc Аннотация.} В работе рассматривается задача локальной параметрической идентифицируемости параметра-функции в системе обыкновенных дифференциальных уравнений. Ранее в этой задаче рассматривался случай совпадения размерностей параметра и решения системы и выделялся определенный класс систем, для которых были получены достаточные условия локальной параметрической идентифицируемости. Мы дополняем данные результаты и рассматриваем более широкий класс систем дифференциальных уравнений, а также рассматриваем случай, когда размерность параметра не больше размерности решения системы. Для обоих случаев мы получаем достаточные условия локальной идентифицируемости параметра-функции по наблюдению решения в конечном наборе точек.

\noindent{\bf Ключевые слова:} дифференциальное уравнение, локальная параметрическая идентифицируемость, бесконечномерный параметр, сингулярное разложение матрицы.
\end{quote}
}


\textbf{1. Введение.}
Задача о параметрической идентификации (определении 
параметров системы по наблюдению решений или функций от
них) --- одна из основных задач прикладной теории 
дифференциальных уравнений. При решении этой задачи
важнейшую роль играет свойство локальной идентифицируемости.
Наличие такого свойства означает, что по наблюдению
решений можно однозначно определить значение
параметров системы в окрестности выделенного
параметра. 

Свойство локальной идентифицируемости
детально изучено в монографии \cite{Bodunov2006AnIntroduction}. Как в монографии \cite{Bodunov2006AnIntroduction}, так и в других исследованиях по данной теме в основном изучался случай конечномерного
параметра. 

Мы можем отметить несколько публикаций, в которых исследуется локальная идентифицируемость бесконечномерного параметра. 
В работе \cite{Pilyugin2024Local} рассматривается несколько формулировок задачи локальной параметрической идентифицируемости для динамических систем с дискретным временем: были получены достаточные условия локальной идентифицируемости па\-ра\-мет\-ра-пос\-ле\-до\-ва\-тель\-нос\-ти для общего случая, когда наблюдения траектории происходят во всех точках траектории, и для случая, когда наблюдения производятся в произвольном счетном множестве точек; также рассмотрен случай линейно возмущенного диффеоморфизма в окрестности гиперболического множества. В статье \cite{Shalgin} изучается случай линейной дискретной системы, зависящей от параметра-последовательности: получены достаточные и необходимые условия локальной параметрической идентифицируемости, также показано, что свойство локальной идентифицируемости топологически типично.

В данной работе мы исследуем постановку задачи, которая также изучалась в статьях \cite{Bodunov2009Local, Volfson2011Local, Pilyugin2023Conditions}. Опишем ее здесь. Рассматривается зависящая от параметра $p$ система обыкновенных дифференциальных уравнений
\begin{equation}
\label{1}
\dot{x}=f(t,x,p), \quad t\in[0,T],
\end{equation}
где $x\in\mathbb{R}^n$, $p\in\mathbb{R}^l$. Предполагается, что вектор-функция $f$, стоящая в правой части системы \eqref{1}, непрерывна по совокупности переменных $(t,x,p)$ и непрерывно дифференцируема по переменным $(x,p)$ в некоторой области $D$ пространства $\mathbb{R}^{1+n+l}$. 

Параметр $p$ является функцией от $t$, принадлежащей некоторому классу $C^1$-гладких функций $\cP$, определенных на некотором интервале $\cI$, содержащем отрезок $[0,T]$. Введем обозначение
$$
    \{p\}=\{p(t):\;t\in[0,T]\}.
$$

Мы фиксируем вектор $x_0\in\mathbb{R}^n$ и обозначаем через $x(t,\{p\})$
решение задачи Коши с начальными данными $(0,x_0)$
для системы \eqref{1}, в которой зафиксирована
параметр-функция $p\in\cP $. Мы предполагаем, что
для любой функции $p\in\cP $ система
$$
    \dot{x}=f(t,x,p(t))
$$
удовлетворяет условиям существования и единственности, решение $x(t,\{p\})$ оп\-ре\-де\-ле\-но на всем промежутке $[0,T]$ и 
$$
    \{(t,x(t,\{p\}),p(t)):\;t\in[0,T]\}\subset D.
$$

Для непрерывной на отрезке $[a,b]$ вектор-функции $q(t)$ будем обозначать
$$
    \|q\|_{a,b}=\max_{t\in[a,b]}|q(t)|.
$$

Будем говорить, что пара параметров $(p_1,p_2)$, $p_1,p_1\in\cP$ различима по наблюдению на множестве $\Theta\subset[0,T]$, если найдется такое $\theta_0\in\Theta$, что
$$
    x(\theta_0,\{p_1\})\neq x(\theta_0,\{p_2\}).
$$
Мы фиксируем эталонный параметр $p_0(t)\in \cP$ и используем следующее определение локальной идентифицируемости параметра $p_0(t)$:
\begin{definition}\label{def_loc_ident_func}
    Система \eqref{1} называется локально параметрически идентифицируемой в классе $\cP$ для параметра $p_0\in\cP$ по наблюдению на множестве $\Theta$, если существует такое $\varepsilon>0$, что если $p\in\cP$ и $0<\|p-p_0\|_{0,T}<\varepsilon$, то пара $(p_0,p)$ различима по наблюдению на множестве $\Theta$.
\end{definition}
Целью является нахождение достаточных условий локальной иден\-ти\-фи\-ци\-руе\-мос\-ти параметра $p_0(t)$ с описанием подходящих класса $\cP$ рас\-смат\-ри\-вае\-мых параметров и множества $\Theta$, на котором наблюдаются решения.

В работе \cite{Bodunov2009Local} было получено достаточное условие локальной идентифицируемости параметра-функции в системе \eqref{1}, когда наблюдение решения происходит в конечной точке $T$. В работе \cite{Volfson2011Local} решена та же задача для случая, когда решение наблюдается в конечном числе точек отрезка $[0,T]$ и при этом семейство $\cP$ рассматриваемых параметр-функций конечномерно. Статья \cite{Pilyugin2023Conditions} развивает идеи предыдущих статей: класс рассматриваемых систем дифференциальных уравнений был определен таким образом, что параметр-функция из некоторого бесконечномерного семейства $\cP$ может быть однозначно определена путем наблюдения решения в конечном числе точек отрезка $[0,T]$. 

В данной работе мы реализуем подход, аналогичный статье \cite{Pilyugin2023Conditions} и дополняем ее результаты. 
В работе \cite{Pilyugin2023Conditions} рассматривался случай системы \eqref{1}, когда $n=l$, т.~е. размерность параметра совпадает с размерностью решения, матрица $\frac{\partial f}{\partial p}(t,x,p)$ непрерывно дифференцируема в области $D$ и определитель $C^1$-гладкой матрицы 
\begin{equation}\label{matrix_D}
    \cD(t)=\frac{\partial f}{\partial p}(t,x(t,\{p_0\}),p_0(t))
\end{equation}
имеет только простые нули (и тогда множество его нулей конечно на отрезке $[0,T]$), наблюдение решения в которых позволяет локально идентифицировать параметр $p_0(t)$. 

В статье \cite{Pilyugin2023Conditions} используется линейное разложение выражения $x(t,\{p\})-x(t,\{p_0\})$, впервые полученное в статье \cite{Volfson2011Local}. Мы приводим его здесь.

Пусть $\tau\in[0,T)$. Обозначим через $Y_{\tau}(t)$ фундаментальную матрицу линейной системы
$$
    \dot{y}=\frac{\partial f}{\partial x}(t,x(t,\{p_0\}),p_0(t))\,y,
$$
удовлетворяющую начальному условию $Y_{\tau}(\tau)=I$, где $I$ --- единичная матрица размера $n\times n$.

Определим для $0\le\tau<\theta\le T$ линейное отображение $\Psi_{\tau,\theta} :\;C([0,T])\to\mathbb{R}^n$ формулой
\begin{equation}\label{Psi}
    \Psi_{\tau,\theta}(p)=\int_{\tau}^\theta Y^{-1}_{\tau}(s)
\frac{\partial f}{\partial p}(s,x(s,\{p_0\}),p_0(s))\,p(s)\,ds.
\end{equation}
Если 
$$
\Delta p(t)=p(t)-p_0(t),
$$
$$
\Delta_p x(t)=x(t,\{p\})-x(t,\{p_0\}),
$$
$$
\Delta_p x(\tau)=0,
$$
то имеет место представление
\begin{equation}\label{rep}
    \Delta_p x(\theta)=Y_\tau(\theta)\Psi_{\tau,\theta}(\Delta p)
    +G_{\tau,\theta}(\Delta p),
\end{equation}
где
\begin{equation}\label{sm}
    \frac{|G_{\tau,\theta}(\Delta p)|}
    {\|\Delta p\|_{\tau,\theta}}\to 0
    ~\mbox{  при  }~\|\Delta p\|_{\tau,\theta}\to 0.
\end{equation}

В данной работе мы рассматриваем два случая системы \eqref{1} и приводим достаточные условия локальной идентифицируемости параметра $p_0(t)$. В разделе 2 мы считаем, что $n=l$ и нули определителя матрицы \eqref{matrix_D} не обязательно простые и наблюдаем решение в этих нулях. В разделе 3 мы считаем, что $l\le n$ и, так как матрица \eqref{matrix_D} не обязательно квадратная, мы вводим условие в некотором смысле аналогичное простоте нулей определителя матрицы \eqref{matrix_D}, и наблюдаем решение в соответствующем конечном множестве точек.

\vskip2mm
{\bf 2. Случай равного количества параметров и уравнений.}
В данном разделе мы предполагаем, что $l=n$ в уравнении \eqref{1}, т.~е. размерность параметра $p$ совпадает с размерностью решения $x$. Мы будем считать, что матрица
$$
    \frac{\partial f}{\partial p}(t,x,p)
$$
класса $C^1$ в области $D$. Как и выше, мы обозначим
$$
    \cD(t)=\frac{\partial f}{\partial p}(t,x(t,\{p_0\}),p_0(t)).
$$

\begin{definition}\label{classK}
    Будем говорить, что система \eqref{1} принадлежит классу ${\cal K}(p_0)$, если для любого такого $\tau\in[0,T]$, что
    \begin{equation}\label{klass_K_D_zeros}
        \det\cD (\tau)=0,
    \end{equation}
    существуют такие $h_{\tau}\neq0$ и натуральное $\nu_{\tau}$, что
    \begin{equation}\label{klass_K_D_zeros_asimptotics}
        \det\cD(\tau+t)=h_{\tau}t^{\nu_{\tau}}+o(t^{\nu_{\tau}}),\quad t\to0.
    \end{equation}
\end{definition}

Так как матрица $\cD(t)$ класса $C^1$, то множество точек $\tau$, для которых выполнены условия \eqref{klass_K_D_zeros}---\eqref{klass_K_D_zeros_asimptotics}, конечно. Обозначим через $0<\tau_1<\tau_2<\ldots<\tau_s<T$ точки интервала $(0,T)$, удовлетворяющие условиям \eqref{klass_K_D_zeros}---\eqref{klass_K_D_zeros_asimptotics}, через $\nu_1,\nu_2,\ldots,\nu_s$ обозначим соответствующие им показатели в \eqref{klass_K_D_zeros_asimptotics}, т.~е. для каждого $i=1,2,\ldots,s$ выполнена асимптотика \eqref{klass_K_D_zeros_asimptotics}, где $\tau=\tau_i$, $\nu=\nu_i$. Положим $\tau_0=0$, $\tau_{s+1}=T$. Если $\det\cD(0)\ne0$, то положим $\nu_{0}=0$, если $\det\cD(T)\ne0$, то положим $\nu_{s+1}=0$. Обозначим
\begin{equation}\label{Theta_K}
    \Theta_{\cal K}(p_0)=\{\tau_0,\tau_1,\dots,\tau_{s+1}\},
\end{equation}
\begin{equation}\label{theta_K}
    \vartheta_{\cal K}(p_0)=(\nu_0,\nu_1,\ldots,\nu_{s+1}).
\end{equation}

Для непрерывной на отрезке $[a,b]$ вектор-функции $q(t)$ будем обозначать
$$
    \|q\|_{i,a,b}=\int_a^b |q(t)|\;dt.
$$

Введем классы $K(\tau_k,\tau_{k+1})$ параметр-функций, зависящих от промежутка $[\tau_k,\tau_{k+1}]$, $k=0,1,\ldots,s$. Зафиксируем $k=0,1,\ldots,s$, обозначим $\tau=\tau_k$, $\theta=\tau_{k+1}$, $\nu=\max(\nu_{k},\nu_{k+1})$. Будем говорить, что функция $p(t)\in C^1({\cal I})$ принадлежит классу $K(\tau,\theta)$, если, во-первых, существуют такие числа $\alpha>0$, $\beta>0$, зависящие от промежутка $[\tau,\theta]$, что
\begin{equation}\label{delta_p_alpha}
    \|\Delta p\|_{i,\tau,\theta}
    \geq\alpha\|\Delta p\|_{\tau,\theta},
\end{equation}
\begin{equation}\label{Psi_delta_p_beta}
    |\Psi_{\tau,\theta}(\Delta p)|
    \geq\beta\|\cD\Delta p\|_{i,\tau,\theta},
\end{equation}
где $\Psi_{\tau,\theta}$ определено в \eqref{Psi} и $\Delta p=p-p_0$, а во-вторых, выполняется одно из следующих четырех условий:
\begin{enumerate}
    \item[($K_1$)] $\det\cD (\tau)=\det\cD (\theta)=0$ и существуют такие $\gamma>0$ и $\kappa>0$, зависящие от промежутка $[\tau,\theta]$, что
    \begin{equation}\label{ineq_D_Delta_p_kappa_1}
        |\cD(t)\Delta p(t)|\le \kappa(t-\tau)^{\nu}\|\Delta p\|_{\tau,\theta},\quad t\in[\tau,\tau+\gamma),
    \end{equation}
    \begin{equation}\label{ineq_D_Delta_p_kappa_2}
        |\cD(t)\Delta p(t)|\le \kappa(\theta-t)^{\nu}\|\Delta p\|_{\tau,\theta},\quad t\in(\theta-\gamma,\theta],
    \end{equation}
    \item[($K_2$)] $\tau=0$, $\det\cD (\tau)\ne0$, $\det\cD (\theta)=0$ и существуют такие $\gamma>0$ и $\kappa>0$, зависящие от промежутка $[\tau,\theta]$, что
    $$
        |\cD(t)\Delta p(t)|\le \kappa(\theta-t)^{\nu}\|\Delta p\|_{\tau,\theta},\quad t\in(\theta-\gamma,\theta],
    $$
    \item[($K_3$)] $\theta=T$, $\det\cD (\tau)= 0$, $\det\cD (\theta)\ne0$ и существуют такие $\gamma>0$ и $\kappa>0$, зависящие от промежутка $[\tau,\theta]$, что
    $$
        |\cD(t)\Delta p(t)|\le \kappa(t-\tau)^{\nu}\|\Delta p\|_{\tau,\theta},\quad t\in[\tau,\tau+\gamma),
    $$
    \item[($K_4$)] $\tau=0$, $\theta=T$, $\det\cD (\tau)\ne 0$, $\det\cD (\theta)\ne0$.
\end{enumerate}

Покажем, что множество $K(\tau,\theta)$ является невырожденным бесконечномерным семейством функций. Так как интеграл от не тождественно равной нулю неотрицательной функции не равен нулю и так как отображение $\Psi_{\tau, \theta}$ ненулевое, то найдутся такие функции $p\neq p_0$, что \eqref{delta_p_alpha}---\eqref{Psi_delta_p_beta} выполнены для некоторых $\alpha>0$, $\beta>0$.

Обозначим через $\|\Psi_{\tau,\theta}\|_i$ норму линейного отображения $\Psi_{\tau, \theta}$, согласованную с нормой аргумента $\|\cdot\|_{i,\tau,\theta}$, т.~е. для любой вектор-функции $q(t)\in C([\tau,\theta])$ выполнено неравенство
\begin{equation*}\label{norm_psi}
    |\Psi_{\tau,\theta}(q)|\leq \|\Psi_{\tau,\theta}\|_i\cdot\|q\|_{i,\tau,\theta}.
\end{equation*}
Ясно, что $\|\Psi_{\tau,\theta}\|_i\leq \max_{t\in[\tau,\theta]}\|Y_{\tau}^{-1}(t)\cD(t)\|,$
где $\|\cdot\|$ --- норма матрицы. Обозначим через $\|\cD\|_{i,\tau,\theta}$ норму линейного оператора $\cD(t)$, действующего в пространстве $C([\tau,\theta])$ с нормой $\|\cdot\|_{i,\tau,\theta}$. Тогда для любой вектор-функции $q(t)\in C([\tau,\theta])$ выполнено неравенство
\begin{equation*}\label{norm_cal_D}
    \|\cD q\|_{i,\tau,\theta}\leq \|\cD\|_{i,\tau,\theta}\cdot \|q\|_{i,\tau,\theta}.
\end{equation*}
Ясно, что $\|\cD\|_{i,\tau,\theta}\leq\max_{t\in[\tau,\theta]}\|\cD(t)\|$.

Множество функций $p$, для которых выполнены неравенства \eqref{ineq_D_Delta_p_kappa_1}---\eqref{ineq_D_Delta_p_kappa_2} бесконечномерно. Действительно, если функция $p$ такова, что $|\Delta p(t)|\le \kappa_0(t-\tau)^{\nu}$ при $t\in[\tau,\tau+\gamma)$, $\kappa_0>0$ (это, например, любая функция, аналитическая в окрестности своего нуля $\tau$ порядка $\nu$) и если $\kappa=\frac{\kappa_0\|\cD\|_{i,\tau,\theta}}{\|\Delta p\|_{\tau,\theta}}$, то
\[
    |\cD(t)\Delta p(t)|\leq \|\cD(t)\|\cdot|\Delta p(t)|
    \le \kappa (t-\tau)^{\nu}\|\Delta p\|_{\tau,\theta}.
\]
Значит, множество функций, удовлетворяющих неравенству в условии $(K_3)$, невырожденное и бесконечномерное. Аналогичный вывод можно сделать для условий $(K_1)-(K_2)$.

Пусть $p\in K(\tau,\theta)$, тогда для $p$ выполнены неравенства \eqref{delta_p_alpha}---\eqref{Psi_delta_p_beta} и, предположим, выполнено условие $(K_3)$. Случаи, когда выполнены условия $(K_1),(K_2)$, рассматриваются аналогично; если выполнено условие $K_4$, то дополнительных ограничений на $p$ не накладывается. 

Если $C^1$-гладкая функция $q(t)$ такова, что для нее выполнены неравенства
\begin{equation}\label{Dq}
    |\cD(t)q(t)|\le \kappa(t-\tau)^{\nu}\|\Delta p\|_{\tau,\theta},\quad t\in[\tau,\tau+\widetilde{\gamma}),
\end{equation}
где $\widetilde{\gamma}\leq \gamma$,
$$
    \|q\|_{\tau,\theta}\leq \frac{1}{2}\|\Delta p\|_{\tau, \theta},
$$
$$
    \|q\|_{i,\tau, \theta}\leq \min\left(\frac{\beta\|\cD\Delta p\|_{i,\tau, \theta}}{2\|\Psi_{\tau,\theta}\|_i},\frac{\|\cD\Delta p\|_{i,\tau, \theta}}{\|\cD\|_{i,\tau,\theta}},\frac{\alpha}{2}\|\Delta p\|_{\tau, \theta}\right),
$$
то функция $p+q$ принадлежит классу $K(\tau,\theta)$ с соответствующими числами $\widetilde{\alpha}\le\frac{\alpha}{4}$, $\widetilde{\beta}\leq\frac{\beta}{4}$, $\widetilde{\gamma}\le\gamma$, $\widetilde{\kappa}\ge4\kappa$. Действительно,
\begin{equation}\label{ineqB1}
\begin{split}
    |\Psi_{\tau, \theta}(\Delta p+q)|
    &\geq|\Psi_{\tau, \theta}(\Delta p)|-|\Psi_{\tau, \theta}(q)|\\&
    \ge\beta\|\cD \Delta p\|_{i,\tau,\theta}-\|\Psi_{\tau,\theta}\|_i\cdot
    \|q\|_{i,\tau, \theta}\\&
    \ge\frac{\beta}{2}\|\cD\Delta p\|_{i,\tau, \theta},
\end{split}
\end{equation}
и, с другой стороны,
\begin{equation}\label{ineqB2}
    \widetilde{\beta}\|\cD(\Delta p+q)\|_{i,\tau,\theta}
    \le \frac{\beta}{4}\left(\|\cD\Delta p\|_{i,\tau,\theta}
    +\|\cD\|_{i,\tau,\theta}\cdot\|q\|_{i,\tau,\theta}\right)
    \le \frac{\beta}{2}\|\cD\Delta p\|_{i,\tau, \theta}.
\end{equation}
Сравнивая неравенства \eqref{ineqB1}---\eqref{ineqB2}, мы получаем неравенство \eqref{Psi_delta_p_beta} для $\Delta p+q$ с числом $\widetilde{\beta}$. 

Неравенство \eqref{delta_p_alpha} для $\Delta p+q$ с числом $\widetilde{\alpha}$ вытекает из следующих двух неравенств:
\begin{equation*}
    \|\Delta p+q\|_{i,\tau, \theta}
    \ge \|\Delta p\|_{i,\tau, \theta}-\|q\|_{i,\tau, \theta}
    \ge \frac{\alpha}{2}\|\Delta p\|_{\tau, \theta},
\end{equation*}
\[
    \widetilde{\alpha}\|\Delta p+q\|_{\tau, \theta}
    \le \frac{\alpha}{4}(\|\Delta p\|_{\tau, \theta}+\|q\|_{\tau, \theta})
    \le \frac{\alpha}{2}\|\Delta p\|_{\tau, \theta},
\]
где мы использовали, что $\|q\|_{\tau,\theta}\leq \|\Delta p\|_{\tau, \theta}$. 

Так как для $q(t)$ выполнено неравенство \eqref{Dq}, то
\[
    |\cD(t)(\Delta p(t)+q(t))|\leq
    2\kappa (t-\tau)^{\nu}\|\Delta p\|_{\tau,\theta}
    =4\kappa (t-\tau)^{\nu}\left(\|\Delta p\|_{\tau,\theta}-\frac12\|\Delta p\|_{\tau,\theta}\right)
\]
\[
    \leq 4\kappa (t-\tau)^{\nu}\left(\|\Delta p\|_{\tau,\theta}-\|q\|_{\tau,\theta}\right)
    \leq \widetilde{\kappa}(t-\tau)^{\nu}\|\Delta p+q\|_{\tau,\theta},
\]
и мы получаем неравенство в условии $(K_3)$ для $\Delta p+q$ с числами $\widetilde{\kappa}$ и $\widetilde{\gamma}$.

\begin{theorem}\label{th_diff_K}
    Предположим, что система \eqref{1} принадлежит классу ${\cal K}(p_0)$ с множествами $\Theta_{\cal K}(p_0)$, $\vartheta_{\cal K}(p_0)$ имеющими вид \eqref{Theta_K}---\eqref{theta_K} тогда система \eqref{1} локально параметрически идентифицируема в классе 
    $$
        \cP_K=\bigcap_{k=0}^{s}K(\tau_k,\tau_{k+1})
    $$
    при параметре $p_0\in\cP_K$ по наблюдениям на множестве $\Theta_{\cal K}(p_0)$.
\end{theorem}

Класс $\cP_{K}$ очевидно невырожденный и бесконечномерный. Для доказательства нам понадобится техническая лемма ниже. Будем обозначать через $\mu(A)$ мининорму произвольной матрицы $A$, т.~е. число
$$
    \mu(A)=\min_{|x|=1}|Ax|.
$$
Хорошо известно, что если матрица $A$ квадратная и невырожденная, то
$$
    \mu(A)=\frac{1}{\|A^{-1}\|},
$$
где $\|\cdot\|$ --- норма матрицы. 

\begin{lemma}\label{lemma_norm_ineq_H}
    Пусть $A(t)$ --- такая $n\times n$ матрица
класса $C^1([a,b])$, что
$$
\det A(t)\neq 0,\quad t\in(a,b),
$$
и если $\det A(c)=0$, $c\in\{a,b\}$, то найдутся такие ненулевое $h_c$ и натуральное $\nu_c$, что
\begin{equation}\label{A_asimptotic}
    \det A(c+t)=h_c|t|^{\nu_{c}}+o(|t|^{\nu_c}),\quad |t|\to0.
\end{equation}

Пусть $r(t)$ --- непрерывная функция на $[a,b]$,
удовлетворяющая неравенству 
$$
\|r\|_{i,a,b}\geq \alpha\|r\|_{a,b}
$$
с некоторым $\alpha>0$.

Предположим, что выполнено одно из следующих условий:

\begin{enumerate}
    \item[1)] $\det A (a)=\det A (b)=0$ и существуют такие $\gamma>0$ и $\kappa>0$, что
    \begin{equation}\label{strong_cond_Ar_0}
        |A(t)r(t)|\le \kappa(t-a)^{\max(\nu_a,\nu_b)}\|r\|_{a,b},\quad t\in[a,a+\gamma),
    \end{equation}
    \begin{equation}\label{strong_cond_Ar_1}
        |A(t)r(t)|\le \kappa(b-t)^{\max(\nu_a,\nu_b)}\|r\|_{a,b},\quad t\in(b-\gamma,b],
    \end{equation}
    \item[2)] $\det A (a)\ne0$, $\det A (b)=0$ и существуют такие $\gamma>0$ и $\kappa>0$, что
    $$
        |A(t)r(t)|\le \kappa(b-t)^{\nu_b}\|r\|_{a,b},\quad t\in(b-\gamma,b],
    $$
    \item[3)] $\det A (a)=0$, $\det A (b)\ne0$ и существуют такие $\gamma>0$ и $\kappa>0$, что
    $$
        |A(t)r(t)|\le \kappa(t-a)^{\nu_a}\|r\|_{a,b},\quad t\in[a,a+\gamma),
    $$
    \item[4)] $\det A (a)\ne0$, $\det A (b)\ne0$.
\end{enumerate}

    Тогда существует такое 
    $\lambda>0$, зависящее только от $A(t)$, $\alpha$, $\gamma$ и $\kappa$, что
    $$
        \|Ar\|_{i,a,b}\geq\lambda \|r\|_{i,a,b}.
    $$
\end{lemma}

\begin{proof}[{\bf Доказательство леммы \ref{lemma_norm_ineq_H}}]
Пусть $[a,b]=[0,1]$. Предположим, что $\det A(t)>0$ при $t\in(0,1)$. Рассмотрим случай, когда выполнено условие 4. 
Так как 
$$
    A(t)^{-1}=\frac{\mathrm{adj}(A(t))}{\det A(t)},
$$
где $\mathrm{adj}(A(t))$ --- присоединенная матрица, составленная из алгебраических дополнений к элементам матрицы $A(t)$, то 
\[
    \mu(A(t))=\frac{1}{\|A(t)^{-1}\|}=\frac{|\det A(t)|}{\|\mathrm{adj}(A(t))\|}.
\]
Значит, мининорма матрицы $A(t)$ непрерывно зависит от $t$ и, так как $\det A(t)>0$, $t\in[0,1]$, мининорма положительна для всех $t\in[0,1]$. Пусть
\[
    \lambda=\min\limits_{t\in[0,1]}\mu(A(t))>0,
\]
тогда
\[
    \int_0^1|A(t)r(t)|\,dt\geq \lambda\int_0^1 |r(t)|dt,
\]
что и требовалось.

Предположим, что выполнено условие 1. Из асимптотики \eqref{A_asimptotic} следует, что для произвольного $\varepsilon\in(0,\min(h_0,h_1))$ найдется такое $\delta\in(0,1)$, что для любых $t\in[0,\delta)$ выполнены неравенства 
\begin{equation}\label{ineq_det_A_0}
    \det A(t)\ge (h_0-\varepsilon)t^{\nu_0},
\end{equation}
\begin{equation}\label{ineq_det_A_1}
    \det A(1-t)\ge (h_1-\varepsilon)t^{\nu_1}.
\end{equation}
Уменьшая $\delta$, если нужно, мы получим, что для любого $c\in(0,\delta)$
\begin{equation*}
    \min_{t\in[c,1-c]}\det A(t)=\min(\det A(c),\det A(1-c)).
\end{equation*}
Положим $M_0=\min(h_0,h_1)-\varepsilon$, $\nu=\max(\nu_0,\nu_1)$, тогда ввиду оценок \eqref{ineq_det_A_0}---\eqref{ineq_det_A_1} получим, возможно, уменьшая $\delta$, что для любых $c\in(0,\delta)$
\begin{equation}\label{ineq_det_A_M_0_c_nu}
    \min_{t\in[c,1-c]}\det A(t)\ge M_0 c^{\nu}.
\end{equation}

Элементы матрицы $A^{-1}(t)$ вычисляются как отношения
соответствующих алгебраических дополнений в матрице $A(t)$ к
определителю матрицы $A(t)$, поэтому из неравенства \eqref{ineq_det_A_M_0_c_nu} следует, что  
существует такое число
$M_1>0$, зависящее только от матрицы $A(t)$, что абсолютные величины
элементов матрицы $A^{-1}(t)$ не превосходят $M_1/c$
для $c\in(0,\delta)$ и $t\in[c,1-c]$.

Таким образом, из оценки \eqref{ineq_det_A_M_0_c_nu} следует, что существует такое число $M>0$, зависящее только от $A(t)$, что
$$
\mu(A(t))=\frac{1}{\|A(t)^{-1}\|}\geq Mc^{\nu},\quad t\in[c,1-c],
$$
для любых $c\in(0,\delta)$.

Тогда
\begin{equation}
\label{Mc001}
\int_c^{1-c}|A(t)r(t)|\;dt\geq Mc^{\nu}
\int_c^{1-c}|r(t)|\;dt
\end{equation}
для $c\in(0,\delta)$.

По нашему предположению,
$$
|r(t)|\leq \|r\|_{0,1}\leq\frac{1}{\alpha}\|r\|_{i,0,1},\quad t\in[0,1].
$$

Уменьшая $\delta$, если нужно, мы получаем из условий \eqref{strong_cond_Ar_0}---\eqref{strong_cond_Ar_1}, что существует такое $N>0$, что
\begin{equation}
\label{T22}
\int_0^{c}|A(t)r(t)|\;dt\leq \frac{Nc^{\nu+1}}{\alpha}\|r\|_{i,0,1},
\end{equation}
\begin{equation}
\label{T23}
\int_{1-c}^{1}|A(t)r(t)|\;dt\leq \frac{Nc^{\nu+1}}{\alpha}\|r\|_{i,0,1},
\end{equation}
для $c\in(0,\delta)$.

Очевидны неравенства
\begin{equation}
\label{rc11}
\int_0^{c}|r(t)|\;dt\leq \frac{c}{\alpha}\|r\|_{i,0,1},~\int_{1-c}^{1}|r(t)|\;dt\leq \frac{c}{\alpha}\|r\|_{i,0,1}
\end{equation}
для $c\in(0,\delta)$.

Следовательно, используя неравенства (\ref{Mc001}), (\ref{T22}) и (\ref{T23}),
мы видим, что
$$
\int_0^{1}|A(t)r(t)|\;dt\geq Mc^{\nu}\int_c^{1-c}|r(t)|\;dt-
\frac{2c^{\nu+1}N}{\alpha}\|r\|_{i,0,1},
$$
а из неравенств (\ref{rc11}) мы получаем оценку
$$
\int_0^{1}|A(t)r(t)|\;dt\geq Mc^{\nu}\left(\|r\|_{i,0,1}- \frac{2c}
{\alpha}\|r\|_{i,0,1}\right)-\frac{2c^{\nu+1}N}{\alpha}\|r\|_{i,0,1}.
$$

Возьмем $c=\min\left(\frac{\alpha M}{4(M+N)},\delta\right)$, тогда правая часть последнего выражения не меньше, чем $\frac{Mc^{\nu}}{2}\|r\|_{i,0,1}$ и в качестве искомого $\lambda$ можно взять $\frac{Mc^{\nu}}{2}$. Случаи, когда выполнены условия 2 или 3, рассматриваются аналогично.

\end{proof}

\begin{proof}[{\bf Доказательство теоремы \ref{th_diff_K}}]
Сначала мы докажем, что существует такое $\delta>0$, что если $p(t)\in \cP_K$,
$$
    \Delta_px(\tau_k)=0,\quad k=1,\ldots,s+1
$$
и $\|\Delta p\|_{0,T}<\delta$, то $p(t)\equiv p_0(t)$ при $t\in[0,T]$. 

Пусть $p(t)\in\cP_K$, тогда для каждого фиксированного $k=0,1,\ldots,s$ функция $p(t)$ лежит в классе $K(\tau_k,\tau_{k+1})$. Из равенств
$$
    \Delta_px(\tau_k)=\Delta_px(\tau_{k+1})=0
$$
и представления \eqref{rep} следует равенство
\begin{equation}\label{22q}
    Y_{\tau_k}(\tau_{k+1})\Psi_{\tau_k,\tau_{k+1}}(\Delta p)=-G_{\tau_k,\tau_{k+1}}(\Delta p).
\end{equation}

Так как матрица
$$
    \frac{\partial f}{\partial x}(t,x(t,\{p_0\}),p_0(t))
$$
ограничена на $[0,T]$, существует такое число $\nu_k>0$, что для 
мининормы $\mu(Y_{\tau_k}(\tau_{k+1}))$ выполнено неравенство
$$
    \mu(Y_{\tau_k}(\tau_{k+1}))\geq \nu_k.
$$

Поэтому
$$
    |Y_{\tau_k}(\tau_{k+1})\Psi_{\tau_k,\tau_{k+1}}(\Delta p)|\geq\nu_k
    |\Psi_{\tau_k,\tau_{k+1}}(\Delta p)|.
$$

Применяя условие \eqref{Psi_delta_p_beta}, мы приходим  к оценке
\begin{equation*}
    |Y_{\tau_k}(\tau_{k+1})\Psi_{\tau_k,\tau_{k+1}}(\Delta p)|
    \geq\beta_k\nu_k
    \|\cD\Delta p\|_{i,\tau_k,\tau_{k+1}}.
\end{equation*}

Из определения класса $K(\tau_k,\tau_{k+1})$ и леммы \ref{lemma_norm_ineq_H} следует, что найдется такое число $\lambda_k>0$, что 
\[
    \|\cD\Delta p\|_{i,\tau_k,\tau_{k+1}}
    \ge\lambda_k\|\Delta p\|_{i,\tau_k,\tau_{k+1}}.
\]
Учитывая последнее неравенство и неравенство \eqref{delta_p_alpha}, мы получаем неравенство
\begin{equation}\label{25q}
    |Y_{\tau_k}(\tau_{k+1})\Psi_{\tau_k,\tau_{k+1}}(\Delta p)|
    \geq\alpha_k\lambda_k\beta_k\nu_k\|\Delta p\|_{\tau_k,\tau_{k+1}}.
\end{equation}

В то же время из оценки \eqref{sm} следует, что по произвольному $\epsilon_k>0$ можно найти такое $\delta_k>0$, что если $\|\Delta p\|_{\tau_{k},\tau_{k+1}}<\delta_k$, то
$$
    |G_{\tau_{k},\tau_{k+1}}(\Delta p)|\leq
    \epsilon_k\left\|\Delta p
    \right\|_{\tau_{k},\tau_{k+1}}.
$$

Сопоставляя эту оценку с равенством \eqref{22q} и оценкой \eqref{25q}, мы видим, что если выбрать $\epsilon_k$ достаточно малым, то $\Delta p(t)\equiv 0$ на промежутке $[\tau_{k},\tau_{k+1}]$. В качестве искомого $\delta$ из начала доказательства мы берем $\min_{k\in\{0,\ldots,s-1\}}\delta_k$ и получаем требуемое.

Теперь докажем локальную параметрическую идентифицируемость в смысле оп\-ре\-де\-ле\-ния \ref{def_loc_ident_func}. Предположим противное, тогда для $\varepsilon=\delta$ найдется такое $p_{\delta}(t)\in\cP_K$, что $0<\|p_{\delta}-p_0\|_{0,T}<\delta$ и $x(\tau_k,\{p_\delta\})=x(\tau_k,\{p_0\})$ для всех $k=1,2,\ldots,s$. Но по доказанному выше мы получаем, что $p\equiv p_0$ на промежутке $[0,T]$, что противоречит выбору $p_{\delta}$, для которого $\|p_{\delta}-p_0\|_{0,T}>0$.

\end{proof}

\vskip2mm
{\bf 3. Количество параметров не больше количества уравнений.}
В этом разделе мы будем считать, что $l\leq n$, т.~е. размерность параметра $p$ не больше размерности решения $x$. Пространства $\R^n$ и $\R^l$, которым принадлежат значения решения $x$ и параметра $p$, соответственно, мы рассматриваем как гильбертовы, снабженные стандартным скалярным произведением и евклидовой нормой. Как и в предыдущем разделе, мы считаем, что матрица
$$
    \frac{\partial f}{\partial p}(t,x,p)
$$
класса $C^1$ в области $D$. Обозначим
$$
    \cD(t)=\frac{\partial f}{\partial p}(t,x(t,\{p_0\}),p_0(t)),
$$
$$
    \cB(t)=\cD(t)^T\cD(t).
$$ 

Матрица $\cB(t)$ размера $l\times l$ положительно полуопределена (неотрицательно определена) для любого $t$, т.~е. для любого вектора $v\in\R^l$ верно неравенство $v^T\cB(t)v=|\cD(t)v|^2\ge0$. Известно, что определитель матрицы равен произведению ее собственных чисел, взятых с учетом кратности. Так как $\cB(t)$ неотрицательно определена, то все ее собственные числа неотрицательны для любого $t$ и $\det \cB(t)\ge0$. 

\begin{definition}\label{classH}
    Будем говорить, что система \eqref{1} принадлежит классу $\cH(p_0)$, если функция $\det \cB(t)$ имеет нули только второго порядка, т.~е. из равенства 
    \begin{equation}\label{B_zero}
        \det \cB(\tau)=0
    \end{equation}
    при некотором $\tau\in[0,T]$ следует, что найдется такое положительное $h_{\tau}>0$, что
    \begin{equation}\label{B_asimptotic}
        \det \cB(\tau+t)=h^2_\tau t^2+o(t^2),\quad t\to 0.
    \end{equation}
\end{definition}

\begin{remark}
    Как будет показано ниже, класс $\cH(p_0)$ является аналогом класса $\cF(p_0)$, определенного в статье \cite{Pilyugin2023Conditions} как класс систем \eqref{1} с $n=l$, для которых функция $\det \cD(t)$ имеет только простые нули. Так как в случае $l\leq n$ матрица $\cD(t)$ не обязательно квадратная и у нее может не быть определителя, то мы вводим квадратную матрицу $\cB(t)=\cD(t)^T\cD(t)$. Если матрица $\cD(t)$ квадратная, то $\det\cB(t)=(\det\cD(t))^2$, и аналогия между классами $\cH(p_0)$ и $\cF(p_0)$ прямая. Заметим также, что ядра матриц $\cD(t)$ и $\cB(t)$ совпадают для любого $t$.
\end{remark}

Множество точек $\tau$, для которых выполнены условия \eqref{B_zero}---\eqref{B_asimptotic}, конечно. Обозначим через $0<\tau_1<\ldots<\tau_s<T$ точки интервала $(0,T)$, удовлетворяющие условиям \eqref{B_zero}---\eqref{B_asimptotic}. Положим $\tau_0=0$, $\tau_{s+1}=T$ и 
\begin{equation}\label{theta_new}
    \Theta(p_0)=\{\tau_0,\dots,\tau_{s+1}\}.
\end{equation}

Введем классы $H(\tau_k,\tau_{k+1})$ параметр-функций, зависящих от промежутка ${[\tau_k,\tau_{k+1}]}$, $k=0,1,\ldots,s$. Зафиксируем $k=0,1,\ldots,s$, обозначим $\tau=\tau_k$ и $\theta=\tau_{k+1}$. Будем говорить, что функция $p(t)\in C^1({\cal I})$ принадлежит классу $H(\tau,\theta)$, если, во-пер\-вых, существуют такие числа $\alpha>0$, $\beta>0$, зависящие от промежутка $[\tau,\theta]$, что
\begin{equation*}
    \|\Delta p\|_{i,\tau,\theta}
    \geq\alpha\|\Delta p\|_{\tau,\theta},
\end{equation*}
\begin{equation*}
    |\Psi_{\tau,\theta}(\Delta p)|
    \geq\beta\|\cD\Delta p\|_{i,\tau,\theta},
\end{equation*}
а во-вторых, выполняется одно из следующих четырех условий:
\begin{enumerate}
    \item[($H_1$)] $\det\cB (\tau)=\det\cB (\theta)=0$ и существуют такие $\gamma>0$ и $\kappa>0$, зависящие от промежутка $[\tau,\theta]$, что
    $$
        |\cD(t)\Delta p(t)|\le \kappa(t-\tau)\|\Delta p\|_{\tau,\theta},\quad t\in[\tau,\tau+\gamma),
    $$
    $$
        |\cD(t)\Delta p(t)|\le \kappa(\theta-t)\|\Delta p\|_{\tau,\theta},\quad t\in(\theta-\gamma,\theta].
    $$
    \item[($H_2$)] $\tau=0$, $\det\cB (\tau)\neq 0$, $\det\cB (\theta)=0$ и существуют такие $\gamma>0$ и $\kappa>0$, зависящие от промежутка $[\tau,\theta]$, что
    $$
        |\cD(t)\Delta p(t)|\le \kappa(\theta-t)\|\Delta p\|_{\tau,\theta},\quad t\in(\theta-\gamma,\theta].
    $$
    \item[($H_3$)] $\theta=T$, $\det\cB (\tau)= 0$, $\det\cB (\theta)\neq 0$ и существуют такие $\gamma>0$ и $\kappa>0$, зависящие от промежутка $[\tau,\theta]$, что
    $$
        |\cD(t)\Delta p(t)|\le \kappa(t-\tau)\|\Delta p\|_{\tau,\theta},\quad t\in[\tau,\tau+\gamma),.
    $$
    \item[($H_4$)] $\tau=0$, $\theta=T$, $\det\cB (\tau)\neq 0$ и $\det\cB(\theta)\neq 0$.
\end{enumerate}

\begin{remark}
    Классы $H(\tau_k,\tau_{k+1})$ невырожденные и бесконечномерные, обоснование чего аналогично рассуждению для классов $K(\tau_k,\tau_{k+1})$, определенных в предыдущем разделе.
\end{remark}

\begin{theorem}\label{th_diff_C}
    Предположим, что система \eqref{1} принадлежит классу $\cH(p_0)$ с множеством $\Theta(p_0)$, имеющим вид \eqref{theta_new}, тогда система \eqref{1} локально параметрически идентифицируема в классе 
    $$
        \cP_H=\bigcap_{k=0}^{s}H(\tau_k,\tau_{k+1})
    $$
    при параметре $p_0\in\cP_H$ по наблюдениям на множестве $\Theta(p_0)$.
\end{theorem}

Для доказательства теоремы \ref{th_diff_C} нам понадобятся некоторые технические построения. Пусть $A$ --- матрица размера $n\times m$.
Как и раньше через $\mu (A)$ мы обозначаем мининорму матрицы $A$:
\[
    \mu(A)=\min\limits_{|v|=1}|Av|,
\]
если $m>n$ или $\ker A\neq\{0\}$, то очевидно $\mu(A)=0$. Так как норма вектора евклидова, то несложно получить другое выражение для мининормы:
\begin{equation}\label{mu_formula}
    \mu(A)=\min\limits_{|v|=1}\sqrt{v^TA^TAv}
    =\min_{1\le i\le m}\sqrt{\lambda_i(A^TA)},
\end{equation}
где $\lambda_i(A^TA)$, $i=1,\ldots,m$ --- собственные числа неотрицательно определенной матрицы $A^TA$.

Число $\sigma\ge0$ называется сингулярным числом матрицы $A$, если найдутся такие единичные вектора $u\in\R^n$, $v\in\R^m$, что 
\[
    Av=\sigma u,
\]
\[
    A^Tu=\sigma v.
\]
Векторы $u,v$ называются, соответственно, левым и правым сингулярными векторами. Сингулярным разложением матрицы $A$ называется разложение вида
\begin{equation*}
    A=U\Sigma V^T,
\end{equation*}
где $U, V$ --- ортогональные матрицы размера $n\times n$ и $m\times m$, соответственно, $\Sigma$ --- матрица размера $n\times m$, на главной диагонали которой стоят сингулярные числа матрицы $A$, а остальные элементы равны нулю. Количество сингулярных чисел у матрицы $A$ с учетом кратности равно $\min(n,m)$, притом количество ненулевых сингулярных чисел равно рангу матрицы $A$. Ненулевые сингулярные числа в точности равны квадратным корням из ненулевых собственных чисел матрицы $A^TA$ или $AA^T$ (см. подробнее \cite[гл. 3]{Horn}).

Таким образом, если матрица $A$ размера $n\times l$, $l\leq n$, то матрица $A$ имеет ровно $l$ сингулярных чисел, и количество нулевых сингулярных чисел равно $l-\rank A=\dim\ker A$. Если $\sigma_i(A)$, $i=1\ldots,l$ --- сингулярные числа матрицы $A$, то согласно \eqref{mu_formula}
\[
    \mu(A)=\min_{1\leq i\leq l}\sigma_i(A).
\]

Если $A(t)$ --- матрица класса $C([a,b])$ размера $n\times l$, $l\le n$, то функция 
$$
    \mu A(t)=\min\limits_{|v|=1}|A(t)v|
$$ 
непрерывна на отрезке $[a,b]$, что следует из соотношения \eqref{mu_formula} и того факта, что для непрерывной мат\-риц-функ\-ции $A(t)$ существуют непрерывные функции $\lambda_i(t)$, $i=1,\ldots,l$, которые для каждого $t\in[a,b]$ являются собственными числами матрицы $A(t)^TA(t)$ (см. теоремы 5.1, 5.2 в \cite[гл. 2]{Kato}).

\begin{lemma}\label{lemma_mu_diff}
    Пусть $A(t)$ --- матрица класса $C^1([a,b])$ размера $n\times l$, $l\leq n$. Обозначим $B(t)=A(t)^TA(t)$. Пусть
    $$
        \det B(t)\neq 0,\quad t\in(a,b),
    $$
    и если $\det B(c)=0$, при некотором $c\in\{a,b\}$, то найдется такое $h_c>0$, что
    \begin{equation}\label{asimptotic_detB}
        \det B(t+c)=h_c^2t^2+o(t^2),\quad |t|\to0.
    \end{equation}
    
    Тогда если $\det B(c)=0$, $c\in\{a,b\}$, то функция
    \[
        \mu A(t)=\min_{|v|=1}|A(t)v|
    \]
    непрерывно дифференцируема в точке $t=c$ и в некоторой ее окрестности. Притом
    \begin{equation*}\label{diff_mu}
    \begin{split}
        &\frac{d (\mu A)}{dt}(a+0)=\frac{h_a}{\sqrt{\Lambda_a}},\mbox{ если }\det B(a)=0,\\&
        \frac{d (\mu A)}{dt}(b-0)=-\frac{h_b}{\sqrt{\Lambda_b}},\mbox{ если }\det B(b)=0,
    \end{split}
    \end{equation*}
    где $\Lambda_c$, $c\in\{a,b\}$ --- произведение ненулевых собственных чисел матрицы $B(c)$. 
\end{lemma}

\begin{Proof}
    Пусть $[a,b]=[0,1]$. Так как матрица $B(t)$ класса $C^1$ и симметрична, то по теореме 6.8 в \cite[гл. 2]{Kato} на отрезке $[0,1]$ существуют непрерывно дифференцируемые функции $\lambda_1(t),\ldots,\lambda_l(t)$, являющиеся для каждого $t$ (не обязательно упорядоченным) набором собственных чисел матрицы $B(t)$. Тогда 
    \begin{equation*}\label{detB_lambdi}
        \det B(t)=\lambda_1(t)\cdot\ldots\cdot\lambda_l(t).
    \end{equation*}

    Рассмотрим точку $t=0$, рассуждение для точки $t=1$ аналогичное. Пусть $\det B(0)=0$. Предположим, что у матрицы $B(0)$ нулевое собственное число имеет кратность как минимум два. Тогда можно считать, что $\lambda_1(0)=\lambda_2(0)=0$. Так как собственные числа матрицы $B(t)$ неотрицательны, и функции $\lambda_1, \lambda_2$ непрерывно дифференцируемы, то $\lambda_1'(0)=\lambda'_2(0)=0$. Значит, $\lambda_1(t)=o(t)$ и $\lambda_2(t)=o(t)$ при $t\to0+$. Тогда
    \[
        \det B(t)=o(t^2), \quad t\to0+,
    \]
    что противоречит соотношению $\det B(t)=h_0^2t^2+o(t^2)$, $t\to0+$, где $h_0>0$. Следовательно, нулевое собственное число матрицы $B(0)$ простое. 
    
    Пусть $\lambda_1(0)=0$ и $\lambda_i(0)>0$, $i=2,\ldots,l$. Найдется такая окрестность $U$ точки $0$ в отрезке $[0,1]$, что $0\in U\subset[0,1]$ и $\lambda_1(t)<\lambda_i(t)$ для любых $t\in U$ и $i=2,\ldots,l$. Множества $\{\lambda_1(t)\}$, $\{\lambda_2(t),\ldots,\lambda_l(t)\}$ не пересекаются для любого $t\in U$. Тогда по предложению 2.6 из \cite{SmoothDecompos} в окрестности $U$ существует $C^1$-гладкое блочное разложение Шура симметричной матрицы $B(t)$, а именно в окрестности $U$ существуют такие $C^1$-гладкие: ортогональная матрица $Q(t)$ размера $l\times l$, симметричная матрица $\Lambda(t)$ размера $(l-1)\times (l-1)$ и скалярная функция $\lambda(t)$, что выполнено соотношение
    \[
        B(t)=Q(t)\begin{pmatrix}
            \lambda(t) & 0 \\
            0 & \Lambda(t)
        \end{pmatrix}
        Q(t)^T,
    \]
    притом собственные числа матрицы $\Lambda(t)$ не совпадают с $\lambda(t)$ при любых $t\in U$, и $\lambda(0)=0$, $\Lambda_0=\det\Lambda(0)>0$. В окрестности $U$ функция $\lambda(t)$ есть наименьшее собственное значение матрицы $B(t)$. Далее в доказательстве везде предполагается, что $t\in U$.

    Ясно, что если $Q(t)=(q(t),~P(t))$, где $q(t)$ размера $l\times 1$, $P(t)$ размера $l\times (l-1)$, то $B(t)q(t)=\lambda(t)q(t)$, так как матрица $Q(t)$ ортогональная, и, значит,
    \begin{equation}\label{lambda_Bq}
        \lambda(t)=q(t)^TB(t)q(t).
    \end{equation}
   Так как $\lambda(t)$ есть наименьшее собственное число матрицы $B(t)$, из соотношения \eqref{mu_formula}, мы получаем, что
    \begin{equation}\label{sqrtlambda}
        \mu A(t)=\sqrt{\lambda(t)}.
    \end{equation}
    Обозначим $w(t)=A(t)q(t)$, тогда ввиду \eqref{lambda_Bq}
    \begin{equation}\label{sqrtlambdaW}
        \sqrt{\lambda(t)}=|w(t)|.
    \end{equation}
    Так как выполнено \eqref{asimptotic_detB}, то при $t\to0+$
    \[
        \frac{\sqrt{\det B(t)}-h_0t}{h_0t}
        =\sqrt{\frac{h_0^2t^2+o(t^2)}{h_0^2t^2}}-1
        \to0;
    \]
    значит, $\sqrt{\det B(t)}=h_0t+o(t)$, $t\to0+$. Так как $\det B(t)=\lambda(t)\cdot \det\Lambda(t)$, то
    \[
        \sqrt{\lambda(t)}=|w(t)|=p_0t+o(t),\quad t\to0+,
    \]
    где $p_0=\frac{h_0}{\sqrt{\Lambda_0}}>0$. Так как вектор $w(t)$ класса $C^1$, то существует конечный предел $w'(0)=\lim_{t\to0+}\frac{w(t)}{t}$. Вектор-функцию $w_0(t)=\frac{w(t)}{t}$ можно доопределить в нуле до непрерывной функции. При этом $w_0(0)\neq0$, так как иначе $w'(0)=0$ и, значит, $w(t)=o(t)$ и $|w(t)|=o(t)$, $t\to0+$, что противоречит тому, что $p_0>0$. 
    
    Так как $w_0(t)$ непрерывна, $t w_0(t)=w(t)\in C^1$ и компоненты вектора $w_0(t)$ не равны нулю одновременно ни для какого $t$, то согласно предложению 3.4 из \cite{SmoothDecompos} функция $|w_0(t)|$ такова, что функция $t|w_0(t)|$ класса $C^1$. Согласно определению функции $w_0(t)$ и соотношениям \eqref{sqrtlambda}---\eqref{sqrtlambdaW}, мы имеем равенства
    \[
        t|w_0(t)|=|w(t)|=\sqrt{\lambda(t)}=\mu A(t);
    \]
    следовательно, функция $\mu A(t)$ непрерывно дифференцируема в нуле и в его окрестности, что и требовалось.  Правая производная функции $\mu A(t)$ в нуле равна
    \[
        p_0=\frac{h_0}{\sqrt{\Lambda_0}},
    \]
    где $\Lambda_0=\det\Lambda(0)$ является произведением ненулевых собственных чисел матрицы $B(0)$. Случай, когда $\det B(1)=0$, рассматривается аналогично. Лемма доказана.
    
\end{Proof}

В доказательстве леммы \ref{lemma_mu_diff} использовались методы, приведенные в доказательстве теоремы 3.6 из \cite{SmoothDecompos} о существовании гладкого сингулярного разложения матрицы, гладко зависящей от параметра.

\begin{lemma}\label{lemma_norm_ineq_2}
    Пусть выполнены условия леммы \ref{lemma_mu_diff}. Пусть $r(t)$ --- непрерывная функция на $[a,b]$,
удовлетворяющая неравенству 
$$
\|r\|_{i,a,b}\geq \alpha\|r\|_{a,b}
$$
с некоторым $\alpha>0$.

Предположим, что выполнено одно из следующих условий:

\begin{enumerate}
    \item[1)] $\det B(a)=\det B(b)=0$ и существуют такие $\gamma>0$ и $\kappa>0$, что
    \begin{equation*}\label{kerA1}
        |A(t)r(t)|\le \kappa(t-a)\|r\|_{a,b},\quad t\in[a,a+\gamma),
    \end{equation*}
    \begin{equation*}\label{kerA2}
        |A(t)r(t)|\le \kappa(b-t)\|r\|_{a,b},\quad t\in(b-\gamma,b];
    \end{equation*}
    \item[2)] $\det B(a)\neq 0$, $\det B (b)=0$ и существуют такие $\gamma>0$ и $\kappa>0$, что
    $$
        |A(t)r(t)|\le \kappa(b-t)\|r\|_{a,b},\quad t\in(b-\gamma,b];
    $$
    \item[3)] $\det B (a)= 0$, $\det B (b)\neq 0$ и существуют такие $\gamma>0$ и $\kappa>0$, что
    $$
        |A(t)r(t)|\le \kappa(t-a)\|r\|_{a,b},\quad t\in[a,a+\gamma);
    $$
    \item[4)] $\det B(a)\neq 0$ и $\det B(b)\neq 0$.
\end{enumerate}

    Тогда существует такое 
    $\lambda>0$, зависящее только от $A(t)$, $\alpha$, $\kappa$ и $\gamma$, что
    $$
        \|Ar\|_{i,a,b}\geq\lambda \|r\|_{i,a,b}.
    $$
\end{lemma}

\begin{Proof}
    Пусть $[a,b]=[0,1]$. Заметим, что ядра матриц $A(t)$ и $B(t)$ совпадают для любого $t$. Пусть выполнено условие 4. Так как $\det B(t)>0$ при $t\in[0,1]$, то $\ker A(t)=\{0\}$ для любых $t\in[0,1]$. Значит, найдется такое $\lambda>0$, что 
    \[
        \min_{t\in[0,1]} \mu A(t)=\min_{t\in[0,1]}\min_{|v|=1}|A(t)v|=\lambda>0.
    \]
    Если данный минимум равен нулю, то он достигается при некоторых $t$ и $v$, что противоречит условию $\ker A(t)=\{0\}$, $t\in[0,1]$. Следовательно,
    \[
        \int_0^1|A(t)r(t)|\,dt\geq \lambda\int_0^1 |r(t)|dt,
    \]
    и мы получаем нужное неравенство.
    
    Предположим, что выполнено условие 1. Из леммы \ref{lemma_mu_diff} известно, что найдется такое $\delta>0$, что функция $\mu A(t)$ непрерывно дифференцируема на $[0,\delta]$ и на ${[1-\delta,1]}$, притом односторонняя производная функции $\mu A(t)$ в нуле положительна, а в единице отрицательна. Значит, уменьшая $\delta$, если нужно, мы можем считать, что функция $\mu A(t)$ монотонна на каждом из отрезков $[0,\delta]$, $[1-\delta,1]$.

    Так как $\det B(t)>0$ на $(0,1)$, то $\mu A(t)>0$ на $(0,1)$ и, уменьшая $\delta$, если нужно, мы получим, что для любого $c\in (0,\delta)$
$$
\min_{t\in[c,1-c]}\mu A(t)=\min(\mu A(c),\mu A(1-c)),
$$ 

Ввиду того, что производная функции $\mu A(t)$ ненулевая в точках $t=0$ и $t=1$, мы можем найти такое число $M>0$, не зависящее от $c\in(0,\delta)$, что
\begin{equation*}
\label{mind1}
\min_{t\in[c,1-c]}\mu A(t)\geq Mc.
\end{equation*}

Тогда
\begin{equation}
\label{Mc2}
\int_c^{1-c}|A(t)r(t)|\;dt\geq Mc
\int_c^{1-c}|r(t)|\;dt
\end{equation}
для $c\in(0,\delta)$.

Формула \eqref{Mc2} аналогична формуле \eqref{Mc001}, дальнейшее рассуждение аналогично стратегии, приведенной в доказательстве леммы \ref{lemma_norm_ineq_H}.

\end{Proof}

Доказательство теоремы \ref{th_diff_C} с учетом леммы \ref{lemma_norm_ineq_2} полностью аналогично доказательству теоремы \ref{th_diff_K}.

\vskip5mm
\nopagebreak\vskip3mm\nopagebreak

{\footnotesize

\par}

\nopagebreak\vskip1.5mm\nopagebreak

\vskip5mm
\begin{otherlanguage}{english}

\noindent{\bf Local identifiability of a parameter function \\in a system of differential equations$^*$\footnote{$^*$The work was performed at the Saint Petersburg Leonhard Euler International Mathematical Institute and supported by the Ministry of Science and Higher Education of the Russian Federation (agreement no. 075–15–2025–343 of April 29, 2025).}}\\[2mm]
\textit{V.\,S.\,Shalgin  
}\\[1.5mm]
{\footnotesize
St. Petersburg State University, 7–9, Universitetskaya nab., St. Petersburg, 199034, Russian Federation,\\
vladimir.shalgin@spbu.ru, st086496@student.spbu.ru }

\vskip3mm

{\small
{\leftskip=7mm\noindent
{\sc Abstract.} In this paper, we consider the problem of local parameter identifiability of a parameter function in a system of ordinary differential equations. Previously, in this problem, the case where the dimensions of a parameter and a solution of a system coincide was considered, and a specific class of systems was identified, for which sufficient conditions for local parametric identifiability were obtained. We extend these results and consider a wider class of systems of differential equations, as well as the case where the dimension of a parameter is less than or equal to the dimension of a solution of a system. In both cases, sufficient conditions are derived for the local identifiability of a parameter function based on observations of a solution at a finite number of points.
\\[1mm]
\textbf{Keywords}: differential equation, local parameter identifiability, infinite-dimensional pa\-ra\-me\-ter, singular matrix decomposition.
\par}

\vskip5mm
\noindent\textbf{References}
\par}
\nopagebreak\vskip3mm\nopagebreak

{\footnotesize
[1] N. A. Bodunov, 
{\it An Introduction to the Theory of Local Parameter Identifiability}, 
St. Petersburg University Press, St. Petersburg, 2006 (in Russian).

[2]
S. Yu. Pilyugin, V. S. Shalgin,
Local parameter identifiability: case of discrete infinite-dimensional parameter, 
{\it Journal of Dynamical and Control Systems}, {\bf 30(14)} (2024).

[3]
V. S. Shalgin, 
Local parameter identifiability of an infinite-dimensional parameter in discrete linear dynamical systems, 
\textit{Differential Equations and Control Processes}, \textbf{2} (2025), 89---98.

[4] 
N. A. Bodunov, G. I. Volfson, 
Local identifiability of systems with a variable parameter, 
{\it Differential Equations and Control Processes}, {\bf 2} (2009), 17---31 (In Russian).

[5]
G. I. Volfson, 
Local parametric identifiability for system with finite family of parameters,
{\it Vestnik Udmurtskogo Universiteta.
Matematika. Mekhanika. Komp'yuternye Nauki}, {\bf 1} (2011), 8---13 (In Russian).

[6]
S. Yu. Pilyugin, V. S. Shalgin, 
Conditions for local parameter identifiability for systems of differential equations with an infinite-dimensional parameter, 
{\it Vestnik St. Petersburg University. Mathematics}, {\bf 56(4)} (2023), 549---558.

[7]
R. A. Horn, C. R. Johnson,
{\it Topics in Matrix Analysis},
Cambridge, 1st. ed., 1991.

[8]
T. Kato,
\textit{Perturbation theory for linear operators},
Heidelberg, Springer Berlin, (1966).
[Russ. ed.: T. Kato, Teoriya vozmushchenij linejnyh operatorov, Moscow, Mir Publ., (1972)]

[9]
L. Dieci, T. Eirola,
On Smooth Decompositions of Matrices,
{\it SIAM Journal on Matrix Analysis and Applications}, 
{\bf 20(3)}, 800-819, (1999).

\par}

\end{otherlanguage}

\end{document}